\documentclass[12pt]{article}
\usepackage{amsrefs}

\usepackage{fullpage}
\usepackage{amsmath,amsthm,amssymb}
\usepackage{mathtools}
\usepackage{mathrsfs}
\usepackage{graphicx}
\usepackage{xcolor}
\usepackage{comment}

\def \bdm{\begin{displaymath}}
	\def \edm{\end{displaymath}}

\def \eps {\varepsilon}
\def \es {\emptyset}

\def \hff{\hat{F}_5}

\renewcommand \b[2] {\binom{#1}{#2}}
\def \ce {\coloneqq}

\renewcommand{\le}{\leqslant}
\renewcommand{\ge}{\geqslant}

\newtheorem{definition}{Definition}

\newtheorem{theorem}[definition]{Theorem}
\newtheorem{lemma}[definition]{Lemma}
\newtheorem{proposition}[definition]{Proposition}
\newtheorem{claim}[definition]{Claim}

\newtheorem{remark}[definition]{Remark}

\renewcommand{\ln}{\log}
\def \lnn {\ln n}
\def \slnn {\sqrt{\ln n}}
\def \lnlnn {\ln \ln n}

\def \slnndn {\slnn/n}

\def \sm {\setminus}

\def \P {\mathbb{P}}

\def \E {\mathbb{E}}

\usepackage{ifthen}
\usepackage{pgf}
\usepackage{tikz}
\usetikzlibrary{graphs}
\usetikzlibrary{positioning}
\usetikzlibrary{calc}

\pgfdeclarelayer{bg}    
\pgfdeclarelayer{fg}    
\pgfsetlayers{bg,main,fg}  

\tikzstyle{every node}=[circle, draw, fill=black!80, inner sep=0pt, minimum width=4pt]

\def\rotateclockwise#1{
	\newdimen\xrw
	\pgfextractx{\xrw}{#1}
	\newdimen\yrw
	\pgfextracty{\yrw}{#1}
	\pgfpoint{\yrw}{-\xrw}
}

\def\rotatecounterclockwise#1{
	\newdimen\xrcw
	\pgfextractx{\xrcw}{#1}
	\newdimen\yrcw
	\pgfextracty{\yrcw}{#1}
	\pgfpoint{-\yrcw}{\xrcw}
}

\def\outsidespacerpgfclockwise#1#2#3{
	\pgfpointscale{#3}{
		\rotateclockwise{
			\pgfpointnormalised{
				\pgfpointdiff{#1}{#2}}}}
}

\def\outsidespacerpgfcounterclockwise#1#2#3{
	\pgfpointscale{#3}{
		\rotatecounterclockwise{
			\pgfpointnormalised{
				\pgfpointdiff{#1}{#2}}}}
}

\def\outsidepgfclockwise#1#2#3{
	\pgfpointadd{#2}{\outsidespacerpgfclockwise{#1}{#2}{#3}}
}

\def\outsidepgfcounterclockwise#1#2#3{
	\pgfpointadd{#2}{\outsidespacerpgfcounterclockwise{#1}{#2}{#3}}
}

\def\outside#1#2#3{
	($ (#2) ! #3 ! -90 : (#1) $)
}

\def\cornerpgf#1#2#3#4{
	\pgfextra{
		\pgfmathanglebetweenpoints{#2}{\outsidepgfcounterclockwise{#1}{#2}{#4}}
		\let\anglea\pgfmathresult
		\let\startangle\pgfmathresult
		
		\pgfmathanglebetweenpoints{#2}{\outsidepgfclockwise{#3}{#2}{#4}}
		\pgfmathparse{\pgfmathresult - \anglea}
		\pgfmathroundto{\pgfmathresult}
		\let\arcangle\pgfmathresult
		\ifthenelse{180=\arcangle \or 180<\arcangle}{
			\pgfmathparse{-360 + \arcangle}}{
			\pgfmathparse{\arcangle}}
		\let\deltaangle\pgfmathresult
		
		\newdimen\x
		\pgfextractx{\x}{\outsidepgfcounterclockwise{#1}{#2}{#4}}
		\newdimen\y
		\pgfextracty{\y}{\outsidepgfcounterclockwise{#1}{#2}{#4}}
	}
	-- (\x,\y) arc [start angle=\startangle, delta angle=\deltaangle, radius=#4]
}

\def\corner#1#2#3#4{
	\cornerpgf{\pgfpointanchor{#1}{center}}{\pgfpointanchor{#2}{center}}{\pgfpointanchor{#3}{center}}{#4}
}

\def\hedgeiii#1#2#3#4{
	\outside{#1}{#2}{#4} \corner{#1}{#2}{#3}{#4} \corner{#2}{#3}{#1}{#4} \corner{#3}{#1}{#2}{#4} -- cycle
}

\def\hedgem#1#2#3#4{
	
	\outside{#1}{#2}{#4}
	\pgfextra{
		\def\hgnodea{#1}
		\def\hgnodeb{#2}
	}
	foreach \c in {#3} {
		\corner{\hgnodea}{\hgnodeb}{\c}{#4}
		\pgfextra{
			\global\let\hgnodea\hgnodeb
			\global\let\hgnodeb\c
		}
	}
	\corner{\hgnodea}{\hgnodeb}{#1}{#4}
	\corner{\hgnodeb}{#1}{#2}{#4}
	-- cycle
}

\def\hgrotate#1{
	\newdimen\x
	\pgfextractx{\x}{#1}
	\newdimen\y
	\pgfextracty{\y}{#1}
	\pgfpoint{-\y}{\x}
}

\def\hgperpr#1#2#3{
	\pgfpointscale{#3}{
		\hgrotate{
			\pgfpointnormalised{
				\pgfpointdiff{#1}{#2}}}}
}

\def\hgaddeperpr#1#2#3{
	\pgfpointadd{#2}{\hgperpr{#1}{#2}{#3}}
}

\def\hgaddsperpr#1#2#3{
	\pgfpointadd{\hgperpr{#1}{#2}{#3}}{#1}
}

\def\hgcorner#1#2#3#4{
	\pgflineto{\hgaddeperpr{#1}{#2}{#4}}
	\pgfmathanglebetweenpoints{#1}{#2}\let\anga\pgfmathresult
	\pgfpatharcto{#4}{#4}{90 + \anga}{0}{0}{\hgaddsperpr{#2}{#3}{#4}}
}

\title{On the Maximum $F_5$-free Subhypergraphs \\ of a Random Hypergraph}

\author{
Igor Araujo\footnotemark[1]
\and
J\'ozsef Balogh\thanks{Department of Mathematics, University of Illinois at Urbana-Champaign, Urbana, Illinois 61801, USA. Research is partially supported by UIUC Campus Research Board RB 22000. E-mail: \texttt{\{igoraa2, jobal, haoranl8\}@illinois.edu}.}  \thanks{Research is partially supported by NSF Grant DMS-1764123 and RTG DMS-1937241, Arnold O. Beckman Research Award (UIUC Campus Research Board RB 22000), the Langan Scholar Fund (UIUC), and the Simons Fellowship.}
\and Haoran Luo\footnotemark[1]
}

\date{}

\begin{document}
\maketitle

\begin{abstract}
Denote by $F_5$ the $3$-uniform hypergraph on vertex set $\{1,2,3,4,5\}$ with hyperedges $\{123,124,345\}$. Balogh, Butterfield, Hu, and Lenz proved that if $p > K \lnn /n$ for some large constant $K$, then every maximum $F_5$-free subhypergraph of $G^3(n,p)$ is tripartite with high probability, and showed that if $p_0 = 0.1\slnn /n$, then with high probability there exists a maximum $F_5$-free subhypergraph of $G^3(n,p_0)$ that is not tripartite. In this paper, we sharpen the upper bound to be best possible up to a constant factor. We prove that if $p > C \sqrt{\ln n} /n $ for some large constant $C$, then every maximum $F_5$-free subhypergraph of $G^3(n, p)$ is tripartite with high probability.
\end{abstract}

\section{Introduction} \label{sec::Int}
In this paper, a (hyper)graph is maximum with respect to a property if it has the maximum number of (hyper)edges among the (hyper)graphs satisfying the given property. Throughout the paper, all logarithms are in base $e$. 

One of the first results in extremal graph theory is Mantel’s Theorem~\cite{mantel1907Pro}, which states that every triangle-free graph on $n$ vertices has at most $\lfloor n^2/4 \rfloor$ edges. Additionally, the complete bipartite graph whose part sizes differ by at most one is the unique maximum triangle-free graph. Later, Tur\'an~\cite{turan1941extremalaufgabe} generalized Mantel's Theorem for all complete graphs. Denote by $K_s$ the complete graph on $s$ vertices and by $T_s(n)$ the complete $s$-partite graph on $n$ vertices where the sizes of the parts differ by at most $1$. Tur\'an's Theorem states that $T_{s-1}(n)$ is the unique maximum $K_s$-free graph on $n$ vertices. Tur\'an's Theorem can also be understood as a property of $K_n$. Namely, it claims that every maximum $K_s$-free subgraph of $K_n$ is $(s-1)$-partite.

Let $G(n, p)$ be the standard binomial model of random graphs, where each edge in $K_n$ is chosen independently with probability $p$. We say that an event occurs \emph{with high pro\-ba\-bi\-li\-ty (w.h.p.)} if its probability goes to $1$ as $n$ goes to infinity. A question related to Tur\'an's~Theo\-rem arises when $G(n,p)$ replaces the role of $K_n$, that is, for what $p = p(n)$ we have that~w.h.p.~every maximum $K_s$-free subgraph of $G(n,p)$ is $(s-1)$-partite. 

This question was first raised by Babai, Simonovits, and Spencer~\cite{babai1990extremal}, who gave an affirmative answer when $p = \frac{1}{2}$ and $s=3$. Later, DeMarco and Kahn~\cite{demarco2015mantel} determined the correct order of $p$ for $s=3$. They~\cite{demarco2015mantel} showed that if $p > K \sqrt{\ln n/n}$ for some large constant $K$, then~w.h.p.~every maximum triangle-free subgraph of $G(n, p)$ is bipartite, while if $p = 0.1\sqrt{\ln n/n}$, then this does not hold~w.h.p. Finally, DeMarco and Kahn~\cite{demarco2015turan} answered this question up to a constant factor for every $s \ge 3$.

\begin{figure}[h!]
	\begin{minipage}{0.50\textwidth}
		\centering
		\caption{The hypergraph $F_5$.}
		\begin{tikzpicture}[scale=2]
			\draw[color=white, opacity = .0, use as bounding box] (0,-1.3) rectangle (2,1.5);
			\node (a) at (0, 1) [label=right:$1$] {};  
			\node (b) at (1, 1) [label=right:$2$] {};
			\node (c) at (2, -1) [label=right:$5$] {};
			\node (d) at (2, 1) [label=right:$3$] {};
			\node (e) at  (2, 0) [label=right:$4$] {};
			
			\begin{pgfonlayer}{bg}
				\draw[color=blue, very thick] \hedgeiii{d}{e}{c}{3mm};
				
				\draw[color=red, very thick] \hedgeiii{a}{b}{d}{2.5mm};
				
				\draw[color=green, very thick] \hedgeiii{a}{b}{e}{2mm};
			\end{pgfonlayer}
		\end{tikzpicture}
	\end{minipage}
	\begin{minipage}{0.50\textwidth}
		\centering
		\caption{The hypergraph $K_4^{-}$.}
		\begin{tikzpicture}[scale=2]
			\draw[color=white, opacity = .0, use as bounding box] (0,-1.3) rectangle (2,1.5);
			\node (a) at (0, 1) [label=right:$1$] {};  
			\node (b) at (1, 1) [label=right:$2$] {};
			\node (d) at (2, 1) [label=right:$3$] {};
			\node (e) at  (2, 0) [label=right:$4$] {};
			\begin{pgfonlayer}{bg}
				\draw[color=blue, very thick] \hedgeiii{d}{e}{b}{3mm};
				
				\draw[color=red, very thick] \hedgeiii{a}{b}{d}{2.5mm};
				
				\draw[color=green, very thick] \hedgeiii{a}{b}{e}{2mm};
			\end{pgfonlayer}
		\end{tikzpicture}
	\end{minipage}
\end{figure}

Similar problems were also  considered for hypergraphs. Denote by $K_4^{-}$ the hypergraph obtained from the complete $3$-uniform hypergraph on four vertices by removing one hyperedge. Let $F_5$, which is often called the \emph{generalized triangle}, be the hypergraph on vertex set $\{1,2,3,4,5\}$ with hyperedges $\{123,124,345\}$. Denote by $S(n)$ the complete $3$-partite $3$-uniform hypergraph on $n$ vertices whose parts have sizes $\left \lfloor n/3 \right \rfloor$, $\left\lfloor (n+1)/3 \right\rfloor$, and $\left\lfloor (n+2)/3 \right\rfloor$, and let $s(n) \ce $ $\left\lfloor n/3 \right\rfloor  \cdot \left\lfloor (n+1)/3 \right\rfloor \cdot \left\lfloor (n+2)/3 \right\rfloor$ be the number of hyperedges in $S(n)$. Bollob{\'a}s~\cite{bollobas1974three} proved that $S(n)$ is the unique maximum $\{K_4^{-}, F_5\}$-free $n$-vertex hypergraph. Frankl and F{\"u}redi~\cite{frankl1983new} proved that, for $n\ge 3000$, the maximum number of hyperedges in an $n$-vertex $F_5$-free $3$-uniform hypergraph is $s(n)$. 

The random version of this theorem was first studied by Balogh, Butterfield, Hu, and Lenz~\cite{balogh2016mantel}. Let $G^3(n,p)$ be the random $3$-uniform hypergraph on vertex set $[n] \ce \{1,2,\ldots, n\}$, where each triple is included with probability $p$ independently of each other. Note that when $p$ is very small, $G^3(n,p)$ itself is tripartite and hence $F_5$-free~w.h.p. Therefore, the interesting case is when $p$ is sufficiently large. In~\cite{balogh2016mantel}, it was proved that if $p > K \ln n /n$ for some large constant $K$, then~w.h.p.~every maximum $F_5$-free subhypergraph of $G^3(n, p)$ is tripartite, and it was conjectured that it suffices to require only $p > C \sqrt{\ln n} /n$ for some large constant $C$. In this paper, we verify this conjecture. This is best possible up to the constant factor, as it was also shown in~\cite{balogh2016mantel} that when $p = 0.1 \sqrt{\ln n}/n$, then~w.h.p.~there is a maximum $F_5$-free subhypergraph of $G^3(n, p)$ that is not tripartite.

\begin{theorem} \label{thm::main}
There exists a constant $C>0$ such that if $p > C \sqrt{\ln n}/n$, then~w.h.p.~every maximum $F_5$-free subhypergraph of $G^3(n,p)$ is tripartite.
\end{theorem}

Our approach will follow the general structure of the proof of the main result of~\cite{balogh2016mantel}. Several key lemmas are improved and adapted for this smaller $p$. 
In particular, in~\cite{balogh2016mantel}, an easier version of codegree concentration was proved using Chernoff's bound with the larger $p$. Here, for the smaller $p$, we need a stronger statement, not only using Chernoff's bound (cf. Lemmas~\ref{lem::lemma5p} and~\ref{lem::lemma5pn}). 
As typical with the probabilistic method, one must fight to avoid applying the union bound when the concentration is not strong enough. This is the most challenging and technical part of the proof, see Remarks~\ref{rem::maindiff} and~\ref{rem::lnlnnintheproof} for more details on that.
We trust that the new ideas used in the proof could be useful for other problems when one has to beat the union bound. 

The rest of the paper is structured as follows. In Section~\ref{sec::Pre}, we will introduce the notation and lemmas needed. In Section~\ref{sec::Pro}, we give the proof of Theorem~\ref{thm::main}.

\section{Preliminaries} \label{sec::Pre}
To improve readability, as it is standard in the literature, we will usually pretend that large numbers are integers to avoid using essentially irrelevant floor and ceiling symbols. We often use the standard upper bound $\b{n}{k} \le (\frac{en}{k})^k$ for binomial coefficients. 
We will use $xy$ to stand for set $\{x,y\}$ and $xyz$ to stand for set $\{x,y,z\}$. We write $x = (1 \pm c)y$ for $(1 - c)y \le x \le (1 + c)y$.

We use $G$ for $G^3(n,p)$ hereinafter and denote by $t(G)$ the number of hyperedges in a maximum tripartite subhypergraph of $G$.

We will always assume that the hypergraphs are on vertex set $[n]=\{1,\ldots,n\}$, so we can identify a hypergraph $H$ by its hyperedges, and $|H|$ stands for the number of hyperedges of $H$. We use $\pi = (V_1,V_2,V_3)$ for a $3$-partition of $[n]$. We say a $3$-partition $\pi$ is \emph{balanced} if every part has size $(1 \pm 10^{-10}) n/3$. Denote by $K_\pi$ the set of triples with exactly one vertex in each part of $\pi$. Let $G_\pi := G \cap K_\pi$. For a hypergraph $H\subseteq G$, let $H_i \ce \{e \in H\;:\; |e \cap V_i| \ge 2\}$ for $i=1,2,3$. Let $H_\pi \ce H \cap K_\pi$ and $\bar{H}_\pi \ce G_\pi \sm H_\pi$. We will call hyperedges in $H_\pi$ the \emph{crossing hyperedges} of $H$, and the hyperedges in $\bar{H}_\pi$ the \emph{missing crossing hyperedges} of $H$.

For a hypergraph $H$, a partition $\pi = (V_1,V_2,V_3)$ of $[n]$, vertices $v,v'\in [n]$, and subsets of vertices $S,T\subseteq [n]$, let
\begin{itemize}
\item $N_{S,T}^H(v) \ce \{ yz : y\in S, z\in T, vyz\in H\}$ be the \emph{link graph} of $v$ between $S$ and $T$,
\item $d_{S,T}^H(v) \ce |N^H_{S,T}(v)|$ be the \emph{degree} of $v$ between $S$ and $T$,
\item $N_{S}^H(v,v') \ce\{z : z \in S, vv'z \in H\}$ be the set of \emph{neighbors} of $v$ and $v'$ in $S$,
\item $d^H_S(v,v') \ce |N_{S}^H (v,v')|$ be the \emph{codegree} of $v$ and $v'$ in $S$, and
\item $L_{S,T}^H(v,v') \ce N^H_{S,T}(v) \cap N^H_{S,T}(v')$ be the \emph{common neighborhood} of $v$ and $v'$ between $S$ and $T$. 
\end{itemize}
When $S$ or $T$ is $[n]$ or $H = G$, we omit to write $S$, $T$, or $G$ when there is no ambiguity. When $S$ or $T$ is $V_i$, we often just use $i$ in the subscript to stand for $V_i$. For example, $d(x)$ is just the degree of $x$ in $G$ and $N^H_{1}(y,z)$ is the neighbors of $y$ and $z$ in $V_1$ in $H$. Finally, define $Q(\pi) \ce \{xy \subset V_1: |L_{2,3}(x,y)|< 0.8p^2n^2/9\}$.

The first proposition we need is the following result from~\cite{balogh2016mantel}, which is a special case of a general transference result of Conlon and Gowers~\cite{conlon2016combinatorial}, and as Samotij observed~\cite{samotij2014stability}, of Schacht~\cite{schacht2016extremal}. It was also proved by the hypergraph container method \cites{balogh2015independent, saxton2015hypergraph}.

\begin{proposition}\label{prop::sta}
For every $\delta > 0$, there exist $\eps > 0$ and $C > 0$ such that if $p > C /n$, then the following statement is true. Let $H$ be a maximum $F_5$-free subhypergraph of $G$ and $\pi$ be a $3$-partition of $[n]$ maximizing $|H_\pi|$. Then, we have that w.h.p.~$\pi$ is balanced, $|H| \ge (2/9 - \eps)p \b{n}{3}$, and $|H \sm H_\pi| \le \delta p n^3$.
\end{proposition}

The following concentration results are also used in~\cite{balogh2016mantel}. Lemma~\ref{lem::che} is the standard Chernoff's bound. Lemmas~\ref{lem::vD} and~\ref{lem::degCro} are standard properties of random hypergraphs, which are direct applications of Lemma~\ref{lem::che} and the union bound.
\begin{lemma} \label{lem::che}
Let $Y$ be the sum of mutually independent indicator random variables, and let $\mu = \E[Y]$. For every $\eps > 0$, we have
\bdm
\P[|Y - \mu| > \eps \mu] < 2e^{-c_\eps \mu},
\edm
where $c_\eps = \min\left\{ -\ln\left(e^\eps(1+\eps)^{-(1+\eps)}\right),\, \eps^2/2\right\}$.
\end{lemma}

\begin{lemma} \label{lem::vD}
For every $\eps >0$, there exists a positive constant $C$ such that if $p > C \ln n / n^2$, then~w.h.p.~for every vertex $v$, we have $d(v) = (1 \pm \eps) pn^2/2$.
\end{lemma}

\begin{lemma} \label{lem::degCro}
For every $\eps >0$, there exists a positive constant $C$ such that if $p > C/n$, then~w.h.p.~for every $3$-partition $\pi = (V_1,V_2,V_3)$ with $|V_2|,|V_3| \ge n/20$ and every vertex $v \in V_1$, we have $d_{2,3}(v) = (1 \pm \eps) p |V_2||V_3|$.
\end{lemma}

We also need the following concentration results.
\begin{lemma} \label{lem::lemma5p}
There exists a positive constant $C$ such that if $p > C\slnn / n$, then~w.h.p.~for every pair of vertices $x,y$, we have $d(x,y) \le pn \slnn / \lnlnn$.
\end{lemma}
\begin{proof}
For every pair of vertices $xy$, the probability that $d(x,y) \ge \frac{pn\slnn}{\lnlnn}$ is at most $\b{n}{\frac{pn\slnn}{\lnlnn}}p^{\frac{pn\slnn}{\lnlnn}}$. There are $\b{n}{2}$ pairs of vertices, so by using a union bound, the probability that there exists a pair of vertices $xy$ such that $d(x,y) \ge \frac{pn\slnn}{\lnlnn}$ is at most
\begin{align*}
&\b{n}{2}\b{n}{pn\frac{\slnn}{\lnlnn}}p^{\frac{pn\slnn}{\lnlnn}} \le n^2 \left(\frac{enp}{pn\frac{\slnn}{\lnlnn}}\right)^{\frac{pn\slnn}{\lnlnn}}\le n^2 \left( \frac{e\lnlnn}{\slnn} \right)^{\frac{C\lnn}{\lnlnn}} \\
\le &\exp\left(2\lnn + \frac{C \ln n}{\lnlnn} \ln \frac{e\lnlnn}{\slnn}\right) = \exp\left( \left(2 + \frac{C\ln (e \lnlnn)}{\lnlnn}-\frac{C}{2}\right) \ln n \right),
\end{align*}
where the last expression is $o(1)$ for sufficiently large $C$.
\end{proof}

\begin{remark} \label{rem::maindiff}
	{\rm Lemma~\ref{lem::lemma5p} shows one of the differences when $p$ is only at least $C \slnn /n$, whereas in~\cite{balogh2016mantel} it is proved that w.h.p.~we have $d(x,y) \le 2pn$ for every pair of vertices $x,y$ when $p> K \lnn /n$. Note that with a direct application of Chernoff's inequality, one can only conclude that $d(x,y) \le pn \slnn$, without the $\lnlnn$ factor. As we will see in Section~\ref{sec::Pro}, this $\lnlnn$ factor plays a vital role in the proof of Theorem~\ref{thm::main} (see Remark \ref{rem::lnlnnintheproof})}.
\end{remark}

\begin{lemma} \label{lem::lemma5pn}
There exists a constant $C>0$ such that if $p > C\sqrt{\ln n}/n$, then~w.h.p.~for every subset $S \subseteq [n]$, we have
\begin{displaymath}
\left|\left\{xy \subseteq [n]\sm S\; : \; d_S(x,y) \ge 3pn\right\}\right| \le n^2e^{-\sqrt{\ln n}}.
\end{displaymath}
\end{lemma}
\begin{proof}
First, consider a fixed set $S$. For every pair of vertices $xy \subseteq [n]\sm S$, the probability that $d_S(x,y) > 3pn$ is at most $q = \b{|S|}{3pn} p^{3pn}$. For different pairs $x_1y_1, x_2y_2 \subseteq [n]\sm S$, random variables $d_S(x_1,y_1)$ and $d_S(x_2,y_2)$ are independent. Hence, given a family of $n^2 e^{- \slnn}$ pairs of vertices in $[n]\sm S$, the probability that $d_S(x,y) \ge 3pn$ for every pair in this family is at most $q^{n^2e^{-\slnn}}$. Then, by a union bound over all the families containing pairs of vertices not in $S$ with size $n^2 e^{-c \slnn}$, we get
$$
\P\left(\left|\left\{xy \subseteq [n]\sm S\; : \; d_S(x,y) \ge 3pn\right\}\right| > n^2e^{-\sqrt{\ln n}}\right)
\le \b{n^2}{n^2e^{-\slnn}}  q^{n^2e^{-\slnn}}.
$$
Finally, using a union bound over all the choices of $S$, the probability of failure is at most
\bdm
2^n\b{n^2}{n^2e^{-\slnn}} \left( \b{n}{3pn} p^{3pn}\right)^{n^2e^{-\slnn}} \le 2^n \left( e^{1+ \slnn} \left(\frac{e}{3}\right)^{3pn} \right)^{n^2e^{-\sqrt{\lnn}}} = o(1). \qedhere
\edm
\end{proof}

In our proof of Theorem~\ref{thm::main}, we will repeatedly use the following lemma to show that for some given vertex $v$ and vertex set $S$, the number of pairs that are in the link graph of $v$ and have a large neighborhood in $S$ is small.

Let $s = s(n,p)$, $r = r(n,p)$, $i = i(n,p)$ be positive integers depending on $n$ and $p$, where $s \le n$, $r \le \b{n}{2}$, and $i\le n$. Let $E_{s,r,i}$ be the event that for every vertex $v \in [n]$ and vertex set $S \subseteq [n]$, where $v \notin S$ and $|S| =s$, we have
$|\{yz \subseteq [n] \sm S: vyz \in G, d_S(y,z) \ge i\}| \le r$, 
see Figure \ref{fig::Esri}. Define $g(p,s,r,i) := n\b{n}{s} \b{n^2}{r} \left(p \b{s}{i}p^{i}\right)^{r}$, and let $O$ be the set of $(s,r,i)$ such that $g(p,s,r,i)= o(n^{-5})$ given $p > C\slnn/n$.
Let $E$ be the event $\bigcap\limits_{(s,r,i) \in O} E_{s,r,i}$.

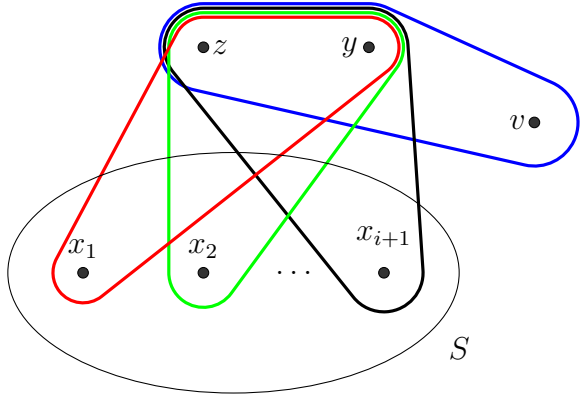
\begin{figure}[h!]
	\centering
	\caption{A pair of vertices $yz$ satisfies that $vyz \in G$ and $d_S(y,z) \ge i$. The event $E_{s,r,i}$ is that for every vertex $v$ and vertex set $S \subseteq [n] \sm \{v\}$ with size $s$, there are at most $r$ such pairs of vertices.}
	\begin{tikzpicture}[scale=2]
		\draw[color=black, opacity = .0, use as bounding box] (0,-0.8) rectangle (2,2);
		\node (v) at (2.5, 1) [label=left:$v$] {};  
		\node (x) at (1.5, 0) [label=above:$x_{i+1}$] {};
		\node (x2) at (0.3, 0) [label=above:$x_2$] {};
		\node (x3) at (-0.5, 0) [label=above:$x_{1}$] {};
		\node (y) at (1.4, 1.5) [label=left:$y$] {};
		\node (z) at (0.3, 1.5) [label=right:$z$] {};
		
		\node at (0.9, 0) [fill=black!0,draw=black!0] {$\ldots$};
		
		\node at (2, -0.5) [fill=black!0,draw=black!0] {$S$};
		\draw (0.5,0) ellipse (1.5 and 0.8) {};
		
		\begin{pgfonlayer}{bg}
			\draw[color=blue, very thick] \hedgeiii{z}{y}{v}{2.9mm};
			
			\draw[color=black, very thick] \hedgeiii{x}{z}{y}{2.6mm};
			
			\draw[color=green, very thick] \hedgeiii{x2}{z}{y}{2.3mm};
			
			\draw[color=red, very thick] \hedgeiii{x3}{z}{y}{2mm};
		\end{pgfonlayer}
	\end{tikzpicture}
	\label{fig::Esri}
\end{figure}

\begin{lemma} \label{lem::lemma12}
E happens with high probability.
\end{lemma}

\begin{proof}
For distinct vertices $v,y,z \in [n]$ and vertex set $S \subset [n]$ where $v,y,z \notin S$ and $|S| = s$, we have $ \P(vyz \in G) = p$ and $\P(d_S(y,z) \ge i) \le \b{s}{i}p^{i}.$
Note that these two events are independent since $v \notin S$. Hence,
$$
\P(vyz \in G, d_S(y,z) \ge i) \le p\b{s}{i}p^{i}.
$$
For fixed $v$ and $S$, events $vyz \in G, d_S(y,z) \ge i$ are independent for different pairs $yz$, so 
$$
\P (|\{yz \subseteq [n] \sm S: vyz \in G, d_S(y,z) \ge i\}| > r)
\le \b{n^2}{r} \left(p \b{s}{i}p^i\right)^{r}.
$$
Then, using a union bound over all choices of $v \in [n]$ and $S \subset [n]$ of size $s$, we get that 
$$
\P(\bar{E}_{s,r,i}) \le n\b{n}{s}\b{n^2}{r} \left(p \b{s}{i}p^i\right)^{r}
= g(p,s,r,i),
$$ 
where $\bar{E}_{s,r,i}$ is the complement of the event $E_{s,r,i}$. Now we have
\bdm
\P(\bar{E}) = \P\Bigg(\bigcup_{(s,r,i) \in O} \bar{E}_{s,r,i}\Bigg) \le \sum_{(s,r,i) \in O} \P(\bar{E}_{s,r,i}) 
\le |O| \cdot o(n^{-5})
\le n\b{n}{2}n\cdot o(n^{-5}) = o(1). \qedhere
\edm
\end{proof}

\begin{remark}
	{\rm Lemmas~\ref{lem::lemma5pn} and~\ref{lem::lemma12} are similar, and one may hope that we can still keep the $e^{-\slnn}$ factor, by proving $(s, pn^2 e^{-\slnn}, 3ps)\in O$ for every $1\le s \le n$. Unfortunately, this is not necessarily true. In the proof of Lemma~\ref{lem::lemma5pn}, we have $n^2 e^{-\slnn}$ in the exponent, which can beat the number of choices of $S$, whereas here $pn^2 e^{-\slnn}$ is not necessarily larger than $n$. However, we have $(s, pn^2/\lnn, 3pn) \in O$, see Claim~\ref{cla::pn}.}
\end{remark}

Finally, we will use the following lemma to give lower bounds on the number of copies of $F_5$. 
For fixed vertex $v$, vertex set $A \subseteq [n] \sm \{v\}$, subset $T$ of $N_{[n]\sm A,[n]\sm A}(v)$, and subset $E$ of $\{vxw \in G\;:\: x \in A\}$ satisfying that for every $x \in A$, there exists $e \in E$ such that $x \in e$,
define
\bdm
K(v,E,A,T) \ce \{xyz: x\in A,\,yz \in T, \textrm{ there exists }e\in E \textrm{ such that } x\in e,\,y\notin e,\,z\notin e \},
\edm
and $G(v,E,A,T) \ce K(v,E,A,T) \cap G$. We have that $\E(|G(v,E,A,T)|) = p |K(v,E,A,T)|$; note that here the randomness in this expectation is of the hyperedges in $\{xyz\;:\: x\in A,\,yz \in T\}$ being or not in $G$. Also note that the events of the hyperedges in $\{xyz\;:\: x\in A,\,yz \in T\}$ being or not in $G$ are independent of the events of hyperedges containing $v$ being in $G$.
For every $xyz \in G(v,E,A,T)$ with $x \in A$, $yz \in T$, we can find a copy of $F_5 = \{yzv, yzx, vxw\}$ in $G$ where $vxw \in E$, see Figure~\ref{fig::GvEAT}. The condition $y\notin e,\,z\notin e$ in the definition of $K(v,E,A,T)$ guarantees that we indeed find an $F_5$ instead of a $K_4^{-}$.

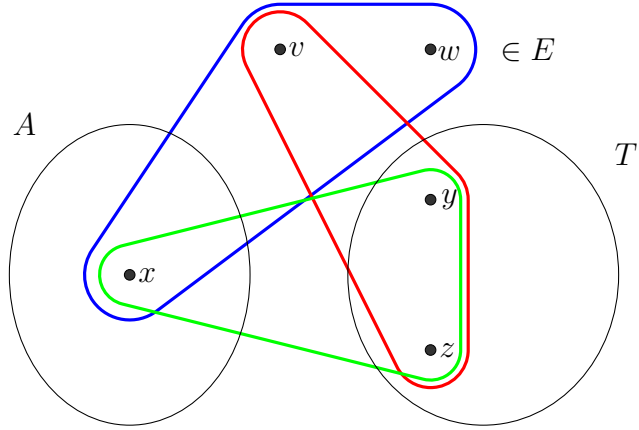
\begin{figure}[h!]
	\centering
	\caption{A copy of $F_5$ from $xyz \in G(v,E,A,T)$.} \label{fig::GvEAT}
	\begin{tikzpicture}[scale=2]
		\draw[color=black, opacity = .0, use as bounding box] (0,0) rectangle (2,3);
		\node (v) at (1, 2.5) [label=right:$v$] {};  
		\node (x) at (0, 1) [label=right:$x$] {};
		\node (y) at (2, 1.5) [label=right:$y$] {};
		\node (z) at (2, 0.5) [label=right:$z$] {};
		\node (w) at  (2, 2.5) [label=right:$w \quad \in E$] {};
		
		\node at (-0.7, 2) [fill=black!0,draw=black!0] {$A$};
		\draw (0,1) ellipse (0.8 and 1) {};
		
		\node at (3.3, 1.8) [fill=black!0,draw=black!0] {$T$};
		\draw (2.35,1) ellipse (0.9 and 1) {};
		
		\begin{pgfonlayer}{bg}
			\draw[color=blue, very thick] \hedgeiii{x}{v}{w}{3mm};
			
			\draw[color=red, very thick] \hedgeiii{v}{y}{z}{2.5mm};
			
			\draw[color=green, very thick] \hedgeiii{x}{y}{z}{2mm};
		\end{pgfonlayer}
	\end{tikzpicture}
\end{figure}

\begin{lemma} \label{lem::lemma9}
For any constants $\eps, \eps_1, \eps_2 >0$, there exists a constant $C>0$ such that if $p > C \slnn / n$, then w.h.p. for every choice of $v,E,A,T$ as above with $|A| \ge \eps_1 n/ \slnn$ and $|T| \ge \eps_2 pn^2$, we have $|G(v,E,A,T)| = (1 \pm \eps ) p |A||T|$.
\end{lemma}
The proof of Lemma~\ref{lem::lemma9} follows the same lines as the proof of Proposition 9 in~\cite{balogh2016mantel}.
\begin{proof}
For a fixed vertex $v$, we first reveal the hyperedges containing $v$. Then, fix a choice of $A$, $T$ and $E$.  
Let $[A,T] \ce \{xyz: x\in A,\,yz \in T\}$. For vertex $x\in A$, let $[x,T] \ce \{xyz : yz \in T\}$, $d^E(x) \ce |\{e \in E: x \in e\}|$, and $T_x \ce \{yz \in T : vxy \in E \textrm{ or } vxz \in E\}$. Note that $[A,T] \supseteq K(v,E,A,T)$ and $d^E(x) \ge 1$ by the definition of $E$. If $d^E(x) \ge 3$, choose $vxw_1, vxw_2,vxw_3$ arbitrarily from $E$. For every $yz \in T$, there exists some $i \in \{1,2,3\}$ such that $y\neq w_i$ and $z \neq w_i$. Hence, by the definition of $K(v,E,A,T)$, we have $[x,T] \subseteq K(v,E,A,T)$. If $d^E(x) \le 2$, then by Lemma~\ref{lem::lemma5p}, we have that w.h.p. $|T_x| \le 2 \cdot \frac{pn\slnn}{\lnlnn}$ and $[x, T \sm T_x] = \{xyz: yz \in T\sm T_x\} \subseteq K(v,E,A,T)$. Therefore,
\bdm
|[A,T]| - |K(v,E,A,T)| \le \sum_{x:\:x \in A,\,d^E(x) \le 2} |T_x| \le |A|\cdot 2 pn \frac{\slnn}{\lnlnn}.
\edm
We have $|[A,T]| = |A||T| \ge |A| \eps_2 pn^2$, so $|K(v,E,A,T)| = (1-o(1)) |A||T|$ with high probability.

Next, we reveal the hyperedges in $[A,T]$. Let $\mu = \E(G(v,E,A,T)) = p|K(v,E,A,T)| = (1-o(1))p|A||T|$. By Lemma~\ref{lem::che}, we have
\bdm
\P(||G(v,E,A,T)| - \mu| > \eps \mu) < 2e^{-c_\eps \mu}.
\edm

Now, we apply the union bound over all possible choices of $(v,E,A,T)$.
We have at most $n$ choices for $v$ and at most $\b{n}{a}$ choices for sets $A$ with size $a$. With high probability we have that for every $a,t >0$, given $|A|=a$ and $|T| = t$, there are at most $2^{apn \frac{\slnn}{\lnlnn}}$ choices for $E$ (by Lemma~\ref{lem::lemma5p}) and at most $\b{pn^2}{t}$ choices for $T$ (by Lemma~\ref{lem::vD}).
By the union bound, the probability that the statement in the lemma does not hold is at most

\begin{align*}
	&\sum_{a \ge \frac{\eps_1 n}{\slnn}} \sum_{t \ge \eps_2 pn^2} n \b{n}{a} 2^{apn \frac{\slnn}{\lnlnn}} \b{pn^2}{t}
	\cdot 2e^{-c_\eps \cdot  atp/2} + o(1)\\
	\le  & \sum_{a \ge \frac{\eps_1 n}{\slnn}} \sum_{t \ge \eps_2 pn^2} n \exp\left( a \ln \frac{en}{a} \right) 2^{apn \frac{\slnn}{\lnlnn}}  \exp\left(t \ln \frac{pn^2}{t}\right)\cdot 2e^{-c_\eps \cdot  atp/2} + o(1)\\
	\le  & 2 \sum_{a \ge \frac{\eps_1 n}{\slnn}} \sum_{t \ge \eps_2 pn^2}
	\exp\left(\lnn + a \ln\left(\frac{e}{\eps_1}\slnn\right) + apn \frac{\slnn \ln 2}{\lnlnn} + t\ln\frac{1}{\eps_2} -  \frac{c_\eps atp}{2}   \right) + o(1) 
\end{align*}\begin{align*}
	\le  & 2 \sum_{a \ge \frac{\eps_1 n}{\slnn}} \sum_{t \ge \eps_2 pn^2}
	\exp\left(-c_\eps atp/4  \right) + o(1)
	\le 2n^3 \exp(-n \slnn) + o(1)
	= o(1). \qedhere
\end{align*}
\end{proof}

\section{Proof of the Main Theorem} \label{sec::Pro}
We first give an outline of the proof.
Recall that for a hypergraph $H$ and a partition $\pi = (V_1,V_2,V_3)$ of $[n]$, we defined in Section~\ref{sec::Pre} that
\begin{align*}
    &\bar{H}_\pi = G_\pi \sm H_\pi, \\
    &H_i = \{e \in H: |e \cap V_i| \ge 2\} \;\; \text{ for $i\in \{1,2,3\}$, \;\;and } \\
    &Q(\pi) = \{uv \subset V_1: |L_{2,3}(u,v)|< 0.8p^2n^2/9\}.
\end{align*}
 
Assume that $H$ is a maximum $F_5$-free subhypergraph of $G^3(n,p)$.
Let $\pi = (V_1,V_2,V_3)$ be a partition of $[n]$ maximizing $|H_\pi|$.  
Let $e = x_1x_2v$ be a hyperedge in $H$ where $x_1,x_2 \in V_1$.
For every $yz \in L_{2,3}(x_1,x_2)$ where $y \neq v$ and $z \neq v$, at least one hyperedge from $\{yzx_1,yzx_2\}$ cannot be in $H$, since otherwise hyperedges $\{yzx_1,yzx_2,x_1x_2v\}$ form a copy of $F_5$. For those $x_1x_2 \notin Q(\pi)$, there are at least $0.8p^2n^2/9 - d(x_1,v)-d(x_2,v) = \Omega(\log n)$ such pairs $yz$. Hence, the existence of $e$ will cause $H$ to lose $\Omega(\log n )$ hyperedges. Since $H$ is maximum, one can expect that $H$ should not contain such hyperedges with more than one vertex in any part of $\pi$, so $H$ is tripartite. Proposition~\ref{prop::lemma13} confirms this idea. 
We also need to handle $Q(\pi)$.
A control over $|Q(\pi)|$ will be given by Proposition~\ref{prop::lemma14}, which states that if $Q(\pi)$ is large, then $t(G)$ will be much larger than $|G_\pi|$. Theorem~\ref{thm::main} will be a simple corollary of Propositions~\ref{prop::lemma13} and~\ref{prop::lemma14}.

For a $3$-uniform hypergraph $H$, the \emph{shadow graph} of $H$ is the graph on the same vertex set, where $xy$ is an edge if and only if there exists another vertex $z$ such that $xyz$ is a hyperedge in $H$.

\begin{proposition} \label{prop::lemma13}
Let $H$ be an $F_5$-free subhypergraph of $G$ and $\pi = (V_1,V_2,V_3)$ be a balanced partition maximizing $|H_\pi|$. Then there exist positive constants $C$ and $\delta$ such that if $p > C \sqrt{\ln n} /n$ and if the following conditions hold:
\begin{enumerate}
  \item  $|H_1|,|H_2|,|H_3| \le \delta p n^3/3$,
  \item the shadow graph of $H_1$ is disjoint from $Q(\pi)$,
\end{enumerate}
then with high probability $|\bar{H}_\pi| \ge 3|H_1|$, where equality is possible only if $H$ is tripartite.
\end{proposition}

\begin{proposition} \label{prop::lemma14}
There exist positive constants $C$ and $\delta$ such that if $p > C \slnndn$ and the $3$-partition $\pi$ is balanced, then with high probability
\bdm
t(G) \ge |G_\pi| + |Q(\pi)| \delta n^2 p^2,
\edm
where equality is possible only if $Q(\pi) = \es$.
\end{proposition}
Based on Propositions~\ref{prop::lemma13} and~\ref{prop::lemma14}, Theorem~\ref{thm::main} easily follows, whose proof is similar to the proof of the main theorem in~\cite{balogh2016mantel}.
\begin{proof}[Proof of Theorem~\ref{thm::main}]
Let $H$ be a maximum $F_5$-free subhypergraph of $G$. We have $|H| \ge t(G)$, since $F_5$ is not tripartite. Let $\pi = (V_1,V_2,V_3)$ be a $3$-partition of $[n]$ maximizing $|H_\pi|$. Without loss of generality, we may assume $|H_1| \ge |H_2|,|H_3|$. Besides, by Proposition~\ref{prop::sta}, we know that w.h.p.~$\pi$ is balanced and $\sum_{i=1}^3 |H_i| \le \delta p n^3/3$, where $\delta$ is a constant smaller than the $\delta$'s in Propositions~\ref{prop::lemma13} and~\ref{prop::lemma14}. Now let 
$$
B(\pi) \ce \{e \in G : \textrm{there exists } uv \in Q(\pi) \textrm{ such that } uv \subset e\}
$$ and $H' \ce H \sm B(\pi)$, so the shadow graph of $H'_1$ is disjoint from $Q(\pi)$. Since $B(\pi)$ consists of only non-crossing hyperedges of $H$, $\pi$ is still a partition maximizing $|H_\pi'|$ and $|H'_1|,|H'_2|,|H'_3| \le \sum_{i=1}^3 |H'_i| \le \sum_{i=1}^3 |H_i| \le \delta p n^3/3$. Hence, $H'$ satisfies the assumptions in Proposition~\ref{prop::lemma13}. Now, we have w.h.p.
\begin{align}
|H|  \le |H_\pi| + 3|H_1| &= |H'_\pi|+ 3|H'_1| + 3|H \cap B(\pi)| \notag \\
    & \le |H'_\pi|+ 3|H'_1| + 3|B(\pi)| \notag \\
    & \le |G_\pi| + 3|B(\pi)| \label{step1} \\
    & \le |G_\pi| + 3|Q(\pi)|pn\slnn/ \lnlnn \label{step2} \\
    & \le |G_\pi| + |Q(\pi)|\delta p^2 n^2 \label{step3} \\
    & \le t(G) \label{step4}.
\end{align}
Here we use Proposition~\ref{prop::lemma13} for (\ref{step1}), Lemma~\ref{lem::lemma5p} for (\ref{step2}), the assumption $p > C \slnn /n$ for (\ref{step3}), and Proposition~\ref{prop::lemma14} for (\ref{step4}).
$|H|$ cannot be strictly smaller than $t(G)$, so all the inequalities must hold with equality. By Propositions~\ref{prop::lemma13} and~\ref{prop::lemma14}, $H'$ is tripartite and $Q(\pi)$ is empty. Then, we conclude that $H = H'$ is tripartite.
\end{proof}

The proof of Proposition~\ref{prop::lemma14} is exactly the same as the one in~\cite{balogh2016mantel}, where it is assumed that $p > K \ln n /n $. It can be easily checked that all the arguments still work verbatim with $p > C \slnndn$, and hence we do not include its proof here. It remains to prove Proposition~\ref{prop::lemma13}. 

Clearly, we can assume that $|H_1| \ge |H_2|, |H_3|$.  
Let $\delta$ be small enough so that the following arguments work, and fix three small positive constants $\eps_1$, $\eps_2$, and $\eps_3$ such that
\bdm
\frac{1}{72\eps_1} \ge 30,\quad \frac{1}{20} \cdot \frac{1}{2\eps_2} \ge 10, \quad \frac{1}{10} - \eps_2 \ge \frac{1}{20}, \quad \frac{100\delta}{\eps_1} \le \eps_3 \le \eps_1, \quad \frac{\frac{1}{4}\eps_1\eps_2}{3\eps_3} \ge 20.
\edm
For example, we can set
\bdm
\delta = 10^{-100},\quad \eps_1 = \frac{1}{3000},\quad \eps_2 = \frac{1}{400},\quad \eps_3 = 10^{-10}.
\edm
Denote by $J$ the induced subgraph of the shadow graph of $H_1$ on the vertex set $V_1$ and use $N^J(x)$, $d^J(x)$ for neighborhood and degree of $x$ in graph $J$, respectively. Call a $4$-set $\{w_1,w_2,y,z\}$ an $\hff$ if $w_1yz$, $w_2yz\in G_\pi$ and there exists $e \in H$ such that $w_1,w_2 \in e \cap V_1$, $y,z\notin e$. Note that $w_1w_2 \in J$ and $\{w_1yz,w_2yz,e\}$ forms a copy of $F_5$ in $G$, so at least one of $w_1yz$ and $w_2yz$ has to be in $\bar{H}_\pi$.

\begin{figure}[h!]
	\centering
	\caption{The set $\{w_1, w_2, y, z\}$ is an $\hff$ if there exists an edge $e \in H_1$ as below.}
	\begin{tikzpicture}[scale=2]
		\draw[color=black, opacity = .0, use as bounding box] (0,-0.3) rectangle (2,2.8);
		\node (a) at (0, 1.5)[label=right:$w_1$] {};  
		\node (b) at (0, 1)[label=right:$w_2$] {};
		\node (c) at (0, 0.5)[label=below:$\qquad e\in H_1$] {};
		\node (d) at (2, 2)[label=below:$y$] {};
		\node (e) at  (2, 0.5)[label=above:$z$] {};
		
		\draw (2.5,2) ellipse (0.8 and 0.7) {};
		\draw (-0.3,1.3) ellipse (0.8 and 0.7) {};
		\draw (2.5,0.5) ellipse (0.8 and 0.7) {};
		
		\node at (-0.8, 1.2) [fill=black!0,draw=black!0] {$V_1$};
		\node at (2.8, 2.2) [fill=black!0,draw=black!0] {$V_2$};
		\node at (2.8, 0.2) [fill=black!0,draw=black!0] {$V_3$};
		
		\begin{pgfonlayer}{bg}
			\draw[color=blue, very thick] \hedgeiii{a}{b}{c}{3mm};
			
			\draw[color=red, very thick] \hedgeiii{a}{d}{e}{2.5mm};
			
			\draw[color=green, very thick] \hedgeiii{b}{d}{e}{2mm};
		\end{pgfonlayer}
	\end{tikzpicture}
\end{figure}
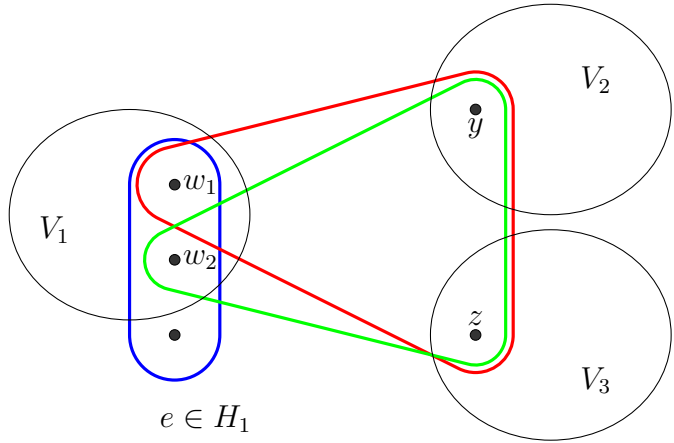

We next count copies of $\hff$ to lower bound the number of missing crossing hyperedges $|\bar{H}_\pi|$, based on the size of (a subgraph of) $J$. First, we prove the following claim, which will be used in the proofs of the next few lemmas.

\begin{claim} \label{cla::over12}
If $p > C\slnn / n$, then~w.h.p.~for every edge $x_1x_2 \in J \setminus Q(\pi)$, there are at least $p^2n^2 /12$ choices of $(y,z)$ where $y \in V_2, z\in V_3$ such that $\{x_1,x_2,y,z\}$ spans an $\hff$.
\end{claim}
\begin{proof}
Since $x_1x_2$ is in $J$, there exists $v_0$ such that $x_1x_2v_0\in H$. Since $x_1x_2$ is not in~$Q(\pi)$, there exist at least $0.8 p^2n^2/9$ choices of $(y,z)$ such that $y\in V_2$, $z \in V_3$, and $x_1yz$ and $x_2yz$ are both in $G_\pi$. By Lemma~\ref{lem::lemma5p}, we know that w.h.p. $d(x_1,v_0)$ is at most $pn\frac{\sqrt{\ln n}}{\lnlnn}$, so there can be at most $pn\frac{\sqrt{\ln n}}{\lnlnn}$ vertices $v$ such that $x_1v_0v \in G_\pi$. Therefore, there are at least $0.8p^2n^2/9 - pn \frac{\sqrt{\ln n}}{\lnlnn} \ge p^2n^2/12$ choices of $(y,z)$ such that $x_1yz,x_2yz \in G_\pi$ and $v_0\not\in \{y,z\}$. For every such $(y,z)$, $\{x_1,x_2,y,z\}$ forms an $\hff$.
\end{proof}

\begin{lemma} \label{lem::lemma17}
Suppose the assumptions of Proposition~\ref{prop::lemma13} hold. Let $J'$ be a subgraph of $J$ and denote by $\Delta(J')$ the maximum degree of $J'$. If $\Delta(J') \le \eps_1 n$, then~w.h.p.~we have
\bdm
|\bar{H}_\pi| \ge 30 pn |J'|.
\edm
If further $\Delta(J') \le \eps_1 n / \slnn$, then~w.h.p.~we have
\bdm
|\bar{H}_\pi| \ge 20pn |J'| \frac{\slnn}{\lnlnn}.
\edm
\end{lemma}

\begin{proof}
For each $wx \in J'$, we get $wx \notin Q(\pi)$ by the assumption of Proposition~\ref{prop::lemma13}, so there are at least $p^2n^2/12$ choices of $(y,z)$ such that $\{w,x,y,z\}$ spans an $\hff$, by Claim~\ref{cla::over12}. Then there are at least $\frac{1}{12} |J'| p^2n^2 = \frac{1}{2}\sum_{x \in V_1} d^{J'}(x) \frac{p^2 n^2}{12}$ copies of $\hff$ in total.

Now consider missing crossing hyperedges $xyz \in \bar{H}_\pi$ with $x \in V_1$. Call $xyz\in \bar{H}_\pi$ \emph{bad} if
\begin{align*}
d^{J'}(x) \le \frac{\eps_1n}{\slnn} \;\;&\text{ and }\;\; |N(y,z)\cap N^{J'}(x)| \ge \frac{pn\lnlnn}{500\slnn}, \;\; \text{ or}\\
\frac{\eps_1n}{\slnn} < d^{J'}(x) \le \eps_1 n \;\;&\text{ and }\;\; |N(y,z)\cap N^{J'}(x)| \ge 3\eps_1 pn.
\end{align*}
Otherwise, call $xyz$ \emph{good}. We will show that the number of copies of $\hff$ that contain a good hyperedge from $\bar{H}_\pi$ is at least $\frac{1}{4}\sum_{x \in V_1} d^{J'}(x) \frac{p^2 n^2}{12}$. For $x \in V_1$, let $n_x$ be the number of copies of $\hff$ that contain vertex $x$ and a bad hyperedge $e \in \bar{H}_\pi$ such that $x \in e$, and let $r_x$ be the number of $(y,z)$ such that $y \in V_2$, $z \in V_3$, and $xyz$ is bad. We will repeatedly use Lemma~\ref{lem::lemma12} for some choices of $(s,r,i)\in O$ to give upper bounds for $r_x$ and then obtain upper bounds for $n_x$.
For every $x \in V_1$, we consider the following two cases.

    \textbf{Case 1:} 
    $d^{J'}(x) \le \frac{\eps_1n}{\slnn}$. If $d^{J'}(x) = 0$, then $n_x$ is trivially $0$.
    For every positive integer $s \le \frac{\eps_1n}{\slnn}$, $r = \frac{pn}{\slnn}s$, and $i =  \frac{pn \lnlnn}{500 \slnn}$, we have $(s,r,i) \in O$ (see Claim~\ref{cla::171} in the Appendix for the proof). Hence, by Lemma~\ref{lem::lemma12}, we get $r_x \le \frac{pn}{\slnn}d^{J'}(x)$. By Lemma~\ref{lem::lemma5p}, we can assume every pair of vertices has codegree at most $pn \frac{\slnn}{\lnlnn}$ in $G$. We conclude
    $$ n_x \le pn\frac{\slnn}{\lnlnn} \cdot \frac{pn}{\slnn}d^{J'}(x)\le \frac{p^2n^2}{100} d^{J'}(x).$$

    \textbf{Case 2:} $\frac{\eps_1n}{\slnn} < d^{J'}(x) \le \eps_1n$.
    For every integer $s \in (\frac{\eps_1n}{\slnn},\,\eps_1n]$, $r = \frac{pn}{500}s$, and $i = 3\eps_1 pn$, we have $(s,r,i) \in O$ (see Claim~\ref{cla::172} in the Appendix for the proof). Hence, by Lemma~\ref{lem::lemma12}, we have $r_x\le \frac{pn}{500}d^{J'}(x)$. By Lemma~\ref{lem::lemma5p}, we can assume every pair of vertices has codegree at most $pn \frac{\slnn}{\lnlnn}$ in $G$. Besides, we have $(s, pn^2/\lnn, 3pn) \in O$ for every $1\le s \le n$ (see Claim~\ref{cla::pn} in the Appendix for the proof), so by Lemma~\ref{lem::lemma12}, there are at most $pn^2 / \ln n$ pairs $(y,z)$ where $y\in V_2$, $z \in V_3$ such that $|N(y,z) \cap N^{J'}(x)| \ge 3pn$. We conclude
    $$n_x \le \frac{pn^2}{\ln n} \cdot pn \frac{\slnn}{\lnlnn} + 3r_xpn
    \le\frac{p^2n^2}{\lnlnn}\cdot \frac{n}{\slnn}+ \frac{3p^2n^2}{500}d^{J'}(x) \le \frac{p^2n^2}{100} d^{J'}(x).$$

Recall that a copy of $\hff$ contains at least one hyperedge from $\bar{H}_\pi$, which is either bad or good. 
Thus, the number of copies of $\hff$ that contain a good hyperedge from $\bar{H}_\pi$ is at least
\begin{equation}
\frac{1}{2}\sum_{x \in V_1} d^{J'}(x) \frac{p^2 n^2}{12} - \sum_{x \in V_1} \frac{p^2n^2}{100}d^{J'}(x) \ge \frac{1}{4}\sum_{x \in V_1} d^{J'}(x) \frac{p^2 n^2}{12} = \frac{1}{24} |J'| p^2n^2. \label{equ::goodhff}
\end{equation}
If $\Delta(J') \le \eps_1 n$, every good missing crossing hyperedge is in at most $3 \eps_1 pn$ copies of $\hff$ estimated in (\ref{equ::goodhff}), so
\bdm
|\bar{H}_\pi| \ge \frac{\frac{1}{24} |J'| p^2n^2}{3\eps_1 pn} \ge \frac{pn |J'|}{72 \eps_1}  \ge 30 pn |J'|.
\edm
If further $\Delta(J') \le \eps_1 n / \sqrt{\lnn}$, then every good missing crossing hyperedge is in at most $\frac{pn  \lnlnn}{ 500\slnn}$ copies of $\hff$ estimated in (\ref{equ::goodhff}), so
\bdm
|\bar{H}_\pi| \ge \frac{\frac{1}{24} |J'| p^2n^2}{ pn \lnlnn / (500\slnn)}\ge \frac{20pn |J'|\slnn}{\lnlnn}. \qedhere
\edm
\end{proof}

Now, we divide the rest of the proof into two cases according to the size of $\bar{H}_\pi$. The notations and calculations are similar in both cases.

\subsection{$|\bar{H}_\pi| \le \delta p n^3 / \lnn$} \label{subsec::bHsmall}
Define
\begin{align*}
S &\ce \{x \in V_1\;:\: d^J(x) \ge \eps_1 n / \slnn\}, \\
S_1 &\ce \{x \in S\;:\: d^H_{2,3}(x) \ge \eps_2pn^2\}, \quad \textrm{and} \quad  S_2 \ce S \sm S_1.
\end{align*}
For the following lemmas, we always assume that $|\bar{H}_\pi| \le \delta p n^3 / \lnn$ and the assumptions of Proposition~\ref{prop::lemma13} hold.

\begin{lemma} \label{lem::lemma151}
With high probability $|S| \le \eps_3 n/\slnn$.
\end{lemma}
\begin{proof}
For each $wx \in J$, we get $wx \notin Q(\pi)$ by the assumption of Proposition~\ref{prop::lemma13}, so there are at least $p^2n^2/12$ choices of $(y,z)$ such that $\{w,x,y,z\}$ spans an $\hff$, by Claim~\ref{cla::over12}. Then there are at least $\frac{1}{12} |J| p^2n^2$ copies of $\hff$ in total. On the other hand, at least one of $wyz$ and $xyz$ must be in $\bar{H}_\pi$. For $xyz$ where $d_{V_1}(y,z) \le 3pn$, $xyz$ can be in at most $3pn$ copies of $\hff$. Hence the number of copies of $\hff$ containing a pair $(y,z) \in V_2\times V_3$ with $d_{V_1}(y,z) \le 3pn$ is at most $3|\bar{H}_\pi|pn$.
By Lemma~\ref{lem::lemma5pn}, there can be at most $n^2 e^{-\sqrt{\ln n}}$ pairs $(y,z)$ with $d_{V_1}(y,z) \ge 3pn$, and by Lemma~\ref{lem::lemma5p}, for every pair $(y,z)$, we can assume $d_{V_1}(y,z) \le d(y,z) \le pn\frac{\sqrt{\ln n}}{\lnlnn}$. Hence the number of copies of $\hff$ containing a pair $(y,z) \in V_2\times V_3$ with $d_{V_1}(y,z) \ge 3pn$ is at most $n^2 e^{-\sqrt{\ln n}}\left(pn\frac{\sqrt{\ln n}}{\lnlnn}\right)^2$.
Therefore, we get

\begin{equation}
n^2 e^{-\sqrt{\ln n}}\left(pn\frac{\sqrt{\ln n}}{\lnlnn}\right)^2+ |\bar{H}_\pi|\cdot 3pn \ge \frac{|J|p^2n^2}{12}. \label{equ::151}
\end{equation}
Note that $n^2 e^{-\sqrt{\ln n}}\left(pn\frac{\sqrt{\ln n}}{\lnlnn}\right)^2 \le \frac{\delta p^2n^4}{\ln n}$. By the assumption $|\bar{H}_\pi| \le \frac{\delta p n^3}{\ln n}$, we have $|\bar{H}_\pi|\cdot 3pn \le \frac{3\delta p^2 n^4}{\ln n}$. Hence, we get $|J| \le \frac{48\delta n^2}{\ln n}$. Every vertex in $S$ has degree at least $\frac{\eps_1n}{\slnn}$ in $J$, so $\frac{\eps_1 n}{ \slnn} |S| \le 2|J| \le \frac{96 \delta n^2 }{ \ln n}$, which confirms that $|S|\le \frac{100\delta}{\eps_1} \cdot \frac{n}{\slnn} \le \frac{\eps_3n}{\slnn}$.
\end{proof}

\begin{lemma} \label{lem::lemma161}
With high probability $|\bar{H}_\pi| \ge 20pn^2 |S_1|$.
\end{lemma}
\begin{proof}
We can assume that $|S_1| > 0$, since otherwise this inequality is trivial. For every vertex $x \in S_1$, define
\bdm
T_x :=  \{yz \;:\: y\in V_2,\, z\in V_3,\, xyz \in H,\, d_{S_1\sm \{x\}}(y,z) \le 3 \eps_3 pn/\slnn\}.
\edm
By Lemma~\ref{lem::lemma151}, we know that w.h.p. $|S_1|\le |S| \le \eps_3 n / \slnn$. For every $s\le \eps_3n/\slnn $, we have $(s, \eps_2 pn^2/2, 3\eps_3 pn / \slnn) \in O$ (see Claim~\ref{cla::161} in the Appendix for the proof), so by Lemma~\ref{lem::lemma12}, there can be at most $\eps_2 pn^2/2$ pairs of $(y,z)$ such that $xyz \in G$ and $d_{S_1 \sm \{x\}}(y,z) \ge 3 \eps_3 pn/\slnn$. By the definition of $S_1$, we have $d^H_{2,3}(x) \ge \eps_2pn^2$ for every $x \in S_1$. Thus, we get $|T_x| \ge  \eps_2 pn^2/2$.

Now, we count the copies of $\hff = \{x,w,y,z\}$ where $x \in S_1$, $w\in N^J(x)$, and $\{y,z\} \in T_x$. Note that $wyz$ must be in $\bar{H}_\pi$, since $xyz \in H$. We have that $|N^{J}(x)| \ge \eps_1 n / \slnn$ by the definition of $S$, and we just confirmed $|T_x| \ge  \eps_2 pn^2/2$. Applying Lemma~\ref{lem::lemma9} by setting $v = x$, $E = \{e \in H_1: x \in e\}$, $A = N^{J}(x)$, and $T = T_x$, we get the number of such copies of $\hff$ is w.h.p. at least
\begin{equation}
\sum_{x \in S_1} \frac{1}{2}p|N^{J}(x)||T_x|\ge \sum_{x \in S_1} \frac{1}{2}p\cdot \frac{\eps_1 n}{\slnn}\cdot \frac{1}{2} \eps_2 pn^2  \ge \frac{\eps_1\eps_2p^2 n^3 |S_1|}{4 \slnn}. \label{equ::161}
\end{equation}
Every $wyz \in \bar{H}_\pi$ can be in at most $3\eps_3pn/\slnn$ copies of $\hff$ evaluated in (\ref{equ::161}), because $x$ is assumed to be in $S_1$ and $d_{S_1}(y,z) \le 3 \eps_3 pn/\slnn$. Therefore,
\bdm
|\bar{H}_\pi| \ge \frac{\frac{1}{4}\eps_1\eps_2p^2 n^3  |S_1|/\slnn}{3\eps_3pn/\slnn} \ge 20 pn^2|S_1|. \qedhere
\edm

\end{proof}

\begin{lemma} \label{lem::lemma181}
With high probability $|\bar{H}_\pi| \ge \frac{1}{20} p n^2 |S_2|$.
\end{lemma}
\begin{proof}
For every vertex $x \in S_2$, we have $d^H_{2,3}(x) < \eps_2 p n^2$ by the definition of $S_2$, but by Lemma~\ref{lem::degCro}, $d_{2,3}(x) \ge pn^2/10$. Thus, there are at least $pn^2/20$ hyperedges in $\bar{H}_\pi$ containing $x$, so $|\bar{H}_\pi| \ge |S_2| pn^2/20$.
\end{proof}

Finally, we deduce Proposition~\ref{prop::lemma13} assuming $|\bar{H}_\pi| \le \delta p n^3 / \lnn$.

\begin{proof} [Proof of Proposition~\ref{prop::lemma13}]
We will show that with high probability $|\bar{H}_\pi| \ge 3|H_1|$, by using the lower bounds in Lemmas~\ref{lem::lemma17},~\ref{lem::lemma161}, and~\ref{lem::lemma181}. Partition $H_1$ into the following three sets.
\begin{itemize}
  \item $H_1(1) = \{e \in H_1 : |e \cap S| \ge 2 \textrm{ or } |e \cap (V_1\sm S)| \ge 2\}$.
  \item $H_1(2) = \{e \in H_1\sm H_1(1) : |e \cap S_1| =1\}$. Hence, $H_1(2)$ contains those hyperedges in $H_1$ with exactly one vertex in $S_1$, one vertex in $V_1 \sm S$, and one vertex in $[n] \sm V_1$.
  \item $H_1(3) = H_1 \sm (H_1(1) \cup H_1(2))$. Hence, $H_1(3)$ contains those hyperedges in $H_1$ with exactly one vertex in $S_2$, one vertex in $V_1 \sm S$, and one vertex in $[n] \sm V_1$.
\end{itemize}
There are three cases needed to be handled.

  \textbf{Case 1:} $3|H_1(1)| \ge |H_1|$.

  Let $J' \ce J[S] \cup J[V_1 \sm S]$, where $J[S]$ and $J[V_1 \sm S]$ are the induced subgraph of $J$ on $S$ and $V_1 \sm S$ separately. By Lemma~\ref{lem::lemma5p}, $|H_1(1)|\le |J'| pn\slnn / \lnlnn$. For every vertex $x \in S$, $d^{J'}(x)$, the degree of $x$ in $J'$, is at most $|S|-1 \le \eps_3 n/\slnn \le \eps_1 n / \slnn$, by Lemma~\ref{lem::lemma151}. For every vertex $x \in V_1\sm S$, we have $d^{J'}(x) \le d^J(x) \le \eps_1 n / \slnn$, by the definition of $S$. Hence, $\Delta(J') \le \eps_1n / \slnn$. Then, by Lemma~\ref{lem::lemma17},
  $$|\bar{H}_\pi| \ge \frac{20pn |J'| \slnn}{\lnlnn} \ge 20 |H_1(1)| \ge 3.3|H_1|.$$

  \textbf{Case 2:} $3|H_1(2)| \ge |H_1|$.

  For every vertex $x \in S_1$, there are at most $2pn^2$ hyperedges in $H_1 \sm H_1(1)$ containing it, by Lemma~\ref{lem::vD}. Hence, $|H_1(2)|\le 2pn^2 |S_1|$. Then, by Lemma~\ref{lem::lemma161},
  $$|\bar{H}_\pi| \ge 20pn^2|S_1| \ge 10 |H_1(2)| \ge 3.3|H_1|.$$

  \textbf{Case 3:} $3|H_1(3)| \ge |H_1|$.

  By the definition of $S_2$, every vertex $x \in S_2$ has $d^H_{2,3}(x) \le \eps_2pn^2$. Recall that $\pi$ maximizes $|H_\pi|$, so $d^H_{1,2}(x), d^H_{1,3}(x) \le \eps_2pn^2$. Hence, $|H_1(3)|\le 2\eps_2pn^2|S_2|$. Then, by Lemma~\ref{lem::lemma181},
  \bdm
  |\bar{H}_\pi| \ge \frac{1}{20}pn^2|S_2| \ge \frac{1}{20}\cdot \frac{|H_1(3)|}{2\eps_2} \ge 10 |H_1(3)| \ge 3.3|H_1|.
  \edm
Thus, we have $|\bar{H}_\pi| \ge 3.3 |H_1| \ge 3 |H_1|$, where the equality is possible only if $|H_1| = 0$. Recalling our assumption that $|H_1|\ge |H_2|,|H_3|$, we have that $|H_1| = 0$ implies $\sum_{i=1}^3 |H_i| = 0$, which means $H$ is tripartite.
\end{proof}

\begin{remark} \label{rem::lnlnnintheproof}
{\rm The $\lnlnn$ factor in Lemma~\ref{lem::lemma5p} plays an important role in the case $3|H_1(1)| \ge |H_1|$. Without this factor, we would need to conclude $|\bar{H}_\pi| \ge 20pn |J'|\slnn$ from Lemma~\ref{lem::lemma17}, which we can only obtain when $\Delta(J') \le \eps_1n / \ln^{1/2+c}n$ for some constant $c>0$. Then, we would need to modify the definition of $S$ to be $\{x \in V_1\;:\: d^J(x) \ge \eps_1 n / \ln^{1/2+c}n\}$. However, the assumption of Lemma~\ref{lem::lemma9} is no longer valid for sets $A$ of smaller size. Therefore, we would not be able to use Lemma~\ref{lem::lemma9} in the proof of Lemma~\ref{lem::lemma161}.}
\end{remark}

\subsection{$|\bar{H}_\pi| > \delta p n^3 / \ln n$} \label{subsec::bHlarge}
All the notation and theorems here are similar to those in Section~\ref{subsec::bHsmall}, so we will just point out the necessary modifications. Define 
\begin{align*}
S' &\ce \{x \in V_1: d^J(x) \ge \eps_1 n\}, \\ 
S'_1 &\ce \{x \in S': d^H_{2,3}(x) \ge \eps_2pn^2\}, \quad \textrm{and} \quad S'_2 \ce S' \sm S'_1. 
\end{align*}
For the following lemmas, we always assume that $|\bar{H}_\pi| > \delta p n^3 / \ln n$ and the assumptions of Proposition~\ref{prop::lemma13} hold. Note that by the assumption in Proposition~\ref{prop::lemma13}, we have $|H_1|, |H_2|,|H_3| \le \delta p n^3/3$. If $|\bar{H}_\pi| \ge 3|H_1|$, then we are done, so we can assume that $|\bar{H}_\pi| < 3|H_1| \le \delta p n^3$.
\begin{lemma} \label{lem::lemma152}
With high probability $|S'| \le \eps_3 n$.
\end{lemma}
\begin{proof}
Inequality (\ref{equ::151}) still holds, and now we get $|J| \le 48 \delta n^2$. Every vertex in $S'$ has degree at least $\eps_1 n$ in $J$, so $\eps_1 n |S'| \le 2 \cdot 48 \delta n^2$, which gives $|S'| \le \frac{100\delta}{\eps_1} n = \eps_3 n$.
\end{proof}

\begin{lemma} \label{lem::lemma162}
With high probability $|\bar{H}_\pi| \ge 20pn^2 |S'_1|$.
\end{lemma}
\begin{proof}
We can assume that $|S'_1| \ge 1$, since otherwise this inequality is trivial. For each $x \in S'_1$, define
$$T'_x := \{yz \;:\: y\in V_2,\, z\in V_3,\, xyz \in H, d_{S'_1 \sm\{x\}}(y,z) \le 3 \eps_3 pn\}.$$
By Lemma~\ref{lem::lemma152}, $|S'_1|\le |S'| \le \eps_3 n$. For every $s\le \eps_3n$, we have $(s, \eps_2pn^2/2, 3\eps_3pn) \in O$ (see Claim~\ref{cla::162} in the Appendix for the proof). Then by Lemma~\ref{lem::lemma12} and the definition of $S'_1$, we get that $|T'_x| \ge \frac{1}{2} \eps_2 pn^2$. Now we count those copies of $\hff = \{x,w,y,z\}$ where $x \in S'_1$, $w\in N^J(x)$, and $\{y,z\} \in T'_x$. By Lemma~\ref{lem::lemma9}, the number of such copies of $\hff$ is at least $\frac{1}{4}\eps_1\eps_2p^2 n^3 |S'_1|$. Every $wyz \in \bar{H}_\pi$ can be in at most $3\eps_3pn$ such copies of $\hff$, because $x$ is assumed to be in $S'_1$ and $d_{S'_1}(y,z) \le 3 \eps_3 pn$. Therefore,
\bdm
|\bar{H}_\pi| \ge \frac{\frac{1}{4}\eps_1\eps_2p^2 n^3  |S'_1|}{3\eps_3pn} \ge 20 pn^2|S'_1|. \qedhere
\edm
\end{proof}

\begin{lemma} \label{lem::lemma182}
With high probability $|\bar{H}_\pi| \ge \frac{1}{20} p n^2 |S'_2|$.
\end{lemma}
\begin{proof}
Exactly the same as the proof of Lemma~\ref{lem::lemma181}.
\end{proof}

We are now able to conclude the proof of Proposition~\ref{prop::lemma13} for the remaining case that $\frac{\delta p n^3}{\lnn} < |\bar{H}_\pi| \le \delta p n^3.$

\begin{proof} [Proof of Proposition~\ref{prop::lemma13}]

Similarly to the proof in Section~\ref{subsec::bHsmall}, we define
$H'_1(1) = \{e \in H_1 : |e \cap S'| \ge 2 \textrm{ or } |e \cap (V_1\sm S')| \ge 2\}$,
$H'_1(2) = \{e \in H_1\sm H'_1(1) : |e \cap S'_1| =1\}$
and $H'_1(3) = H'_1 \sm (H'_1(1) \cup H'_1(2))$.

We still split the proof into three cases. The cases $3|H'_1(2)| \ge |H_1|$ and $3|H'_1(3)| \ge |H_1|$ follow with the same proof as in Section~\ref{subsec::bHsmall}. For the case $3|H'_1(1)| \ge |H_1|$, now let $J' = J[S'] \cup J[V_1 \sm S']$. By Lemmas~\ref{lem::lemma5p} and~\ref{lem::lemma5pn},
\begin{equation}
|H'_1(1)| \le |J'|\cdot 3pn + n^2 e^{- \slnn} \cdot pn\frac{\slnn}{\lnlnn}. \label{equ::132}
\end{equation}
Recall we have the assumptions that $3|H_1| \ge |\bar{H}_\pi|$, so $3|H'_1(1)| \ge|H_1| \ge \delta p n^3 / (3\lnn)$. Hence by (\ref{equ::132}), we have that $0.99|H'_1(1)| \le |J'|\cdot 3pn$. Then, by Lemma~\ref{lem::lemma152} and the definition of $S'$, we get that $\Delta(J') \le \eps_1n$. Finally, by Lemma~\ref{lem::lemma17}, $|\bar{H}_\pi| \ge 30pn |J'| \ge 9.9|H'_1(1)| \ge 3.3 |H_1|$.

Thus, similarly as in Section~\ref{subsec::bHsmall}, we get $|\bar{H}_\pi| \ge 3.3 |H_1| \ge 3|H_1|$, where the equality is possible only if $H$ is tripartite.
\end{proof}

{\bf Acknowledgement:} We thank the referees for their useful comments and careful reading of the manuscript.

\bibliography{reference.bib}

\begin{bibdiv}
\begin{biblist}

\bib{babai1990extremal}{article}{
      author={Babai, L{\'a}szl{\'o}},
      author={Simonovits, Mikl{\'o}s},
      author={Spencer, Joel},
       title={Extremal subgraphs of random graphs},
        date={1990},
     journal={Journal of Graph Theory},
      volume={14},
      number={5},
       pages={599\ndash 622},
}

\bib{balogh2016mantel}{article}{
      author={Balogh, J{\'o}zsef},
      author={Butterfield, Jane},
      author={Hu, Ping},
      author={Lenz, John},
       title={Mantel's theorem for random hypergraphs},
        date={2016},
     journal={Random Structures \& Algorithms},
      volume={48},
      number={4},
       pages={641\ndash 654},
}

\bib{balogh2015independent}{article}{
      author={Balogh, J{\'o}zsef},
      author={Morris, Robert},
      author={Samotij, Wojciech},
       title={Independent sets in hypergraphs},
        date={2015},
     journal={Journal of the American Mathematical Society},
      volume={28},
      number={3},
       pages={669\ndash 709},
}

\bib{bollobas1974three}{article}{
      author={Bollob{\'a}s, B{\'e}la},
       title={Three-graphs without two triples whose symmetric difference is
  contained in a third},
        date={1974},
     journal={Discrete Mathematics},
      volume={8},
      number={1},
       pages={21\ndash 24},
}

\bib{conlon2016combinatorial}{article}{
      author={Conlon, David},
      author={Gowers, Timothy},
       title={Combinatorial theorems in sparse random sets},
        date={2016},
     journal={Annals of Mathematics},
       pages={367\ndash 454},
}

\bib{demarco2015mantel}{article}{
      author={DeMarco, Bobby},
      author={Kahn, Jeff},
       title={Mantel's theorem for random graphs},
        date={2015},
     journal={Random Structures \& Algorithms},
      volume={47},
      number={1},
       pages={59\ndash 72},
}

\bib{demarco2015turan}{article}{
      author={DeMarco, Bobby},
      author={Kahn, Jeff},
       title={Tur{\'a}n's theorem for random graphs},
        date={2015},
     journal={arXiv:1501.01340},
}

\bib{frankl1983new}{article}{
      author={Frankl, Peter},
      author={F{\"u}redi, Zolt{\'a}n},
       title={A new generalization of the {E}rd{\H{o}}s-{K}o-{R}ado theorem},
        date={1983},
     journal={Combinatorica},
      volume={3},
      number={3},
       pages={341\ndash 349},
}

\bib{mantel1907Pro}{article}{
      author={Mantel, Willem},
       title={Problem 28},
        date={1907},
     journal={Wiskundige Opgaven},
       pages={60\ndash 61},
}

\bib{samotij2014stability}{article}{
      author={Samotij, Wojciech},
       title={Stability results for random discrete structures},
        date={2014},
     journal={Random Structures \& Algorithms},
      volume={44},
      number={3},
       pages={269\ndash 289},
}

\bib{saxton2015hypergraph}{article}{
      author={Saxton, David},
      author={Thomason, Andrew},
       title={Hypergraph containers},
        date={2015},
     journal={Inventiones mathematicae},
      volume={201},
      number={3},
       pages={925\ndash 992},
}

\bib{schacht2016extremal}{article}{
      author={Schacht, Mathias},
       title={Extremal results for random discrete structures},
        date={2016},
     journal={Annals of Mathematics},
       pages={333\ndash 365},
}

\bib{turan1941extremalaufgabe}{article}{
      author={Tur{\'a}n, Paul},
       title={Eine {E}xtremalaufgabe aus der {G}raphentheorie},
        date={1941},
     journal={Mat. Fiz. Lapok},
      volume={48},
      number={61},
       pages={436\ndash 452},
}

\end{biblist}
\end{bibdiv}

\section*{Appendix: Final computations}

Recall that $g(p,s,r,i) = n\b{n}{s} \b{n^2}{r} \left(p \b{s}{i}p^{i}\right)^{r}$ and $O$ is the set of $(s,r,i)$ such that $g(p,s,r,i)= o(n^{-5})$ given $p > C\slnn/n$. In this appendix, we give the proof for the claims made in the proof of Theorem~\ref{thm::main} that certain $(s,r,i)$ is in $O$.

\begin{claim} \label{cla::171}
For every positive integer $s \le \frac{\eps_1n}{\slnn}$, $r = \frac{pn}{\slnn}s$, and $i =  \frac{pn \lnlnn}{500 \slnn}$, we have $(s,r,i) \in O$.
\end{claim}
\begin{proof}
We have
\begin{align}
g(p,s,r,i)
&= n \b{n}{s} \b{n^2}{\frac{pn}{\slnn}s} \left( p \b{s}{\frac{pn \lnlnn}{500 \slnn}} p^{\frac{pn \lnlnn}{500 \slnn}} \right)^{\frac{pn}{\slnn}s} \notag \\
&\le n \left( \frac{en}{s}\right)^s \left( \frac{epn^2 \slnn}{ pns} \left( \frac{500eps\slnn}{pn \lnlnn} \right)^{\frac{pn \lnlnn}{500\slnn}} \right)^{\frac{pn}{\slnn}s} \notag \end{align}\begin{align}
& \le n \left( \frac{en}{s}\right)^s \left( \frac{en \slnn}{ s} \left( \frac{500es\slnn}{n \lnlnn} \right)^{\frac{pn \lnlnn}{500\slnn}} \right)^{\frac{pn}{\slnn}s}. \label{g}
\end{align}
There are three cases depending on $s$.

    \textbf{Case 1:} $1\le s \le \sqrt{n}$.
    In this case, (\ref{g}) is at most
    \bdm
    n (en)^s \left( en \slnn \left( \frac{\slnn}{\sqrt{n}} \right)^{\frac{C \lnlnn}{500}} \right)^{Cs}
    \le n (en)^s \left(\frac{1}{n}\right)^{Cs} = o(n^{-5}).
    \edm

    \textbf{Case 2:} $\sqrt{n} < s < \frac{n}{\lnn}$.
    In this case, we have $\frac{n}{s} \ge \lnn > \left(\frac{500e\slnn}{\lnlnn}\right)^2$ and then
    $$
    \left( \frac{500es\slnn}{n \lnlnn} \right)^{\frac{pn \lnlnn}{500\slnn}} \le
    \left( \frac{s}{n} \cdot \frac{500e \sqrt{\ln n}}{\ln \ln n} \right)^{\ln \ln n} \le  \left( \frac{s}{n} \right)^{\frac{1}{2}\lnlnn}  \le \left(\frac{s}{n}\right)^{100}.
    $$
    
    Therefore, (\ref{g}) is at most
    \begin{align*}
    n \left( \frac{en}{s}\right)^s &\left( \frac{en \slnn}{ s} \left(\frac{s}{n}\right)^{100} \right)^{\frac{pn}{\slnn}s}
    \le n \left( \frac{en}{s}\right)^s \left( \frac{s}{n}\right)^{\frac{pn}{\slnn}s}
    \\
    &\le n \left( \frac{en}{s}\right)^s \left( \frac{s}{n}\right)^{Cs}
    \le n \left( \frac{s}{n}\right)^{s}
    \le n \left( \frac{1}{\lnn}\right)^{\sqrt{n}}
    = o(n^{-5}).
    \end{align*}

    \textbf{Case 3:} $\frac{n}{\lnn} \le s \le \frac{\eps_1 n}{\slnn}$.
    In this case, (\ref{g}) is at most
    \begin{align*}
        n(e \lnn)^s  \left(e \ln^{3/2} n \left( \frac{500e \eps_1}{ \lnlnn} \right)^{\frac{C\lnlnn}{500}} \right)^{\frac{pn}{\slnn}s}
    \le & n(e \lnn)^s \left( \left(\frac{1}{\lnlnn}\right)^{\lnlnn} \right)^{Cs}\\
    \le n \left(\frac{1}{\lnn}\right)^{s} & \le  n \left(\frac{1}{\lnn}\right)^{\frac{n}{\lnn}} = o(n^{-5}). \qedhere
    \end{align*}

\end{proof}

\begin{claim} \label{cla::172}
For every positive integer $s \in \big( \frac{\eps_1n}{\slnn}, \,\eps_1n \big]$, $r = \frac{pn}{500}s$, and $i = 3\eps_1 pn$, we have $(s,r,i) \in O$.
\end{claim}

\begin{proof}
We have that $g(p,s,r,i)$ is
\begin{align*}
& n \b{n}{s} \b{n^2}{\frac{pn}{500}s} \left( p \b{s}{3\eps_1 pn} p^{3\eps_1pn} \right)^{\frac{pn}{500}s}
\le n \left( \frac{en}{s}\right)^s \left( \frac{500pn^2}{ pns} \left( \frac{eps}{3\eps_1pn} \right)^{3\eps_1pn} \right)^{\frac{pn}{500}s}\\
&= n \left( \frac{en}{s}\right)^s \left( \frac{500n}{s} \left( \frac{es}{3\eps_1n} \right)^{3\eps_1pn} \right)^{\frac{pn}{500}s}
\le n \left( \frac{e\slnn}{\eps_1}\right)^s \left( \frac{500\slnn}{\eps_1} \left( \frac{e}{3} \right)^{3\eps_1C\slnn} \right)^{Cs\slnn}\\
&\le n \left( \frac{e\slnn}{\eps_1}\right)^s \left( 2^{-\slnn} \right)^{Cs\slnn} \le n2^{-s} \le n2^{-\frac{\eps_1n}{\slnn}} = o(n^{-5}).\qedhere
\end{align*}
\end{proof}

\begin{claim} \label{cla::pn}
For every positive integer $s \le n$, we have $\left(s, \frac{pn^2}{\lnn}, 3pn\right) \in O$.
\end{claim}
\begin{proof}
We have that $g(p,s,r,i)$ is
\begin{align*}
&n \b{n}{s} \b{n^2}{\frac{pn^2}{\lnn}} \left(p\b{s}{3pn}p^{3pn} \right)^{\frac{pn^2}{\lnn}}
\le n2^n \b{n^2}{\frac{pn^2}{\lnn}} \left( p\b{n}{3pn} p^{3pn}\right)^{\frac{pn^2}{\lnn}}  \\
&\le n2^n  \left( \frac{epn^2\lnn}{pn^2} \left(\frac{enp}{3pn}\right)^{3pn} \right)^{\frac{pn^2}{\lnn}}
\le n2^n\left(e \lnn \left(\frac{e}{3}\right)^{3pn} \right)^{\frac{pn^2}{\lnn}}\\
&
\le n2^n\left(e \lnn \left(\frac{e}{3}\right)^{3C\slnn} \right)^{\frac{Cn}{\slnn}}
\le n2^n\left(\left(\frac{e}{3}\right)^{\slnn} \right)^{\frac{Cn}{\slnn}}
=   n2^n \left(\frac{e}{3}\right)^{Cn} = o(n^{-5}). \qedhere
\end{align*}
\end{proof}

\begin{claim} \label{cla::161}
For every positive integer $s\le \frac{\eps_3n}{\slnn} $, we have $\left( s, \frac{\eps_2 pn^2}{2}, \frac{3\eps_3 pn}{\slnn} \right) \in O$.
\end{claim}
\begin{proof}
We have that $g(p,s,r,i)$ is
\begin{align*}
    & n\b{n}{s} \b{n^2}{\frac{1}{2}\eps_2pn^2} \left( p \b{s}{\frac{3\eps_3pn}{\slnn}}p^{\frac{3\eps_3pn}{\slnn}} \right)^{\frac{1}{2}\eps_2pn^2}
    \le n\b{n}{\frac{\eps_3n}{\slnn}} \b{n^2}{\frac{1}{2}\eps_2pn^2} \left( p \b{\frac{\eps_3n}{\slnn}}{\frac{3\eps_3pn}{\slnn}}p^{\frac{3\eps_3pn}{\slnn}} \right)^{\frac{1}{2}\eps_2pn^2} \\
    &\le n \left(\frac{en\slnn}{\eps_3n}\right)^{\frac{\eps_3n}{\slnn}} \left( \frac{2epn^2}{\eps_2pn^2} \left(\frac{e\eps_3pn}{3\eps_3pn}\right)^{\frac{3\eps_3pn}{\slnn}}\right)^{\frac{1}{2}\eps_2pn^2} \\
    &= n \left(\frac{e\slnn}{\eps_3}\right)^{\frac{\eps_3n}{\slnn}} \left( \frac{2e}{\eps_2} \left(\frac{e}{3}\right)^{\frac{3\eps_3pn}{\slnn}}\right)^{\frac{1}{2}\eps_2pn^2} \\
    &\le n \left( \frac{e\slnn}{\eps_3}\right)^{\frac{\eps_3 n}{\slnn}} 2^{-\frac{1}{2}\eps_2pn^2} \le n \left( \frac{e\slnn}{\eps_3}\right)^{\frac{\eps_3 n}{\slnn}} 2^{-\frac{1}{2}\eps_2 Cn\slnn} =o(n^{-5}).\qedhere
\end{align*}
\end{proof}

\begin{claim} \label{cla::162}
For every positive integer $s\le \eps_3n$, we have $\left( s, \frac{\eps_2pn^2}{2}, 3\eps_3pn \right) \in O$.
\end{claim}
\begin{proof}
We have that $g(p,s,r,i)$ is
\begin{align*}
    &n\b{n}{s} \b{n^2}{\frac{1}{2}\eps_2pn^2} \left( p \b{s}{3\eps_3pn}p^{3\eps_3pn} \right)^{\frac{1}{2}\eps_2pn^2}
    \le n\b{n}{\eps_3n} \b{n^2}{\frac{1}{2}\eps_2pn^2} \left( p \b{\eps_3n}{3\eps_3pn}p^{3\eps_3pn} \right)^{\frac{1}{2}\eps_2pn^2} \\
    &\le n \left(\frac{en}{\eps_3n}\right)^{\eps_3n} \left( \frac{2epn^2}{\eps_2pn^2} \left(\frac{e\eps_3pn}{3\eps_3pn}\right)^{3\eps_3pn}\right)^{\frac{1}{2}\eps_2pn^2}
    = n \left(\frac{e}{\eps_3}\right)^{\eps_3n} \left( \frac{2e}{\eps_2} \left(\frac{e}{3}\right)^{3\eps_3pn}\right)^{\frac{1}{2}\eps_2pn^2} \\
    &\le n2^{-n\slnn}
    =o(n^{-5}).\qedhere
\end{align*}
\end{proof}

\end{document}